\documentclass[final,onefignum,onetabnum]{siamart171218}

\usepackage{lipsum}
\usepackage{amsfonts}
\usepackage{graphicx}
\usepackage{epstopdf}
\usepackage{algorithmic}
\usepackage{xcolor}

\usepackage{here} 

\ifpdf
  \DeclareGraphicsExtensions{.eps,.pdf,.png,.jpg}
\else
  \DeclareGraphicsExtensions{.eps}
\fi

\newsiamremark{remark}{Remark}
\newsiamremark{hypothesis}{Hypothesis}
\crefname{hypothesis}{Hypothesis}{Hypotheses}
\newsiamthm{claim}{Claim}

\headers{Laws of exponential functionals of renewal-reward processes}{J. Akahori, C. Constantinescu, Y. Imamura and HH. Pham}

\title{ON DISTRIBUTIONAL AND ASYMPTOTIC RESULTS FOR
EXPONENTIAL FUNCTIONALS IN RENEWAL-REWARD
PROCESSES DESCRIBING INSURANCE RISK MODELS
\thanks{
\funding{This research was partially supported by the European Union's Seventh Framework Programme for research, technological development and demonstration under grant agreement no 318984 - RARE.
}
}
}
\author{J. Akahori\thanks{Ritsumeikan University, Kusatsu, Japan
  (\email{akahori@se.ritsumei.ac.jp}, 
  ).}
\and C. Constantinescu\thanks{University of Liverpool, Liverpool, UK
  (\email{C.Constantinescu@liverpool.ac.uk }).}
\and 
Y. Imamura\thanks{Kanazawa Unikversity, Kanazawa, Japan
(\email{imamuray@se.kanazawa-u.ac.jp}).}
\and 
HH. Pham\thanks{
International University,  Vietnam National University Hochiminh city, Vietnam
  (\email{phha@hcmiu.edu.vn}).}
}

\usepackage{amsopn}

\ifpdf
\hypersetup{
pdftitle={ON DISTRIBUTIONAL AND ASYMPTOTIC RESULTS FOR
EXPONENTIAL FUNCTIONALS OF RENEWAL-REWARD
PROCESSES DESCRIBING INSURANCE RISK MODELS
},
  pdfauthor={Akahori, Constantinescu, Imamura, Pham}
}
\fi

\begin{document}

\maketitle
\begin{abstract}

Inspired by the double-debt problem in Japan where the mortgagor has to pay the remaining loan even if their house was destroyed by a catastrophic event, we model the lender's cash flow, by an exponential functional of a renewal-reward process. We propose an insurance add-on to the loan repayments and analyse the asymptotic behavior of the distribution of the first hitting time, which represents  the probability of full repayment. We show that the finite-time probability of full loan repayment converges exponentially fast to the infinite-time one. In a few concrete scenarios, we calculate the exact form of the infinite-time probability and the corresponding premiums.

\end{abstract}



\section{Introduction}


We develop a mathematical model for mortgage
loans so that we can estimate/measure the risks of the lenders. 
The model is 
inspired by a framework designed
to hedge the so-called 
``double-debt problem", 
introduced by Ohgaki \cite{ADEH1}.  For analyzing the risks, we employ methods from risk theory, culminating in an analysis of solutions integro-differential equations with boundary conditions.\\

{\bf The Motivation.}
After the 2011 Great East Japan Earthquake, a lot of people who lost their houses are still kept under the due of their mortgage loan, which made their recovery rather difficult. 
It is commonly referred to as the double-debt problem. 
As Japan is exposed to the risk of
further big earthquakes, 
Ohgaki \cite{ADEH1} proposed a practical framework within 
the regime of the Japanese financial system,
where the mortgage loan is combined with a marketized earthquake insurance, like a CAT bond. This paper mathematically formalizes  Ohgaki's \cite{ADEH1} proposed scheme. 
The proposed model will be in continuous time as the mortgage payments of individual mortgagors are not likely to be paid at the same time in a given payment period. However, in order to marketize this insurance-mortgage-security, we need to know the risk exposure. \\

 {\bf The Model.}
%
Specifically, we consider an initial loan $u$ given to a cohort of borrowers (as mortgages) that are paying it back continuously at a constant rate $c>0$. 
Although the loans are fixed term, as mentioned before, the assumption of continuous payments is reasonable, due to the cohort effect, 
meaning that new customers are coming in as others are leaving the programme. Furthermore, 
we consider that disasters occur at random times $T_i$ and after each disaster $i$ only a ratio $e^{-X_i}$ of borrowers are left to repay the loan. 
Thus, as time passes and disasters occur we have fewer and fewer borrowers paying back the loan (at the same rate $c$). 
We define as default/ruin the event that the (cash contribution) "process $U_t$ never reaching $u$", meaning that the borrowers will never fully repay the loan. Thus, the probability of default, $\phi (u),$ describes the probability of the process $U_t$ never reaching level $u$, while the survival probability, $\psi(u),$ defines the probability of the first crossing of the level $u$, or {\it probability of full loan repayment.}

Let $ (T_i, X_i) $, $ i=1, 2 \cdots $ be a marked point process, with $T_0 =0$. 
In our model, the cash flow process of the 
mortgage loan at time $t$
is given by
%
\begin{equation}\label{eq:model}
\begin{split}
U_{t}&= c \sum_{n=0}^\infty 1_{ [T_n , T_{n+1}) } (t) 
\{ (t-T_n) e^{-\sum_{i=1}^n X_i}
+ \sum_{i=1}^{n} (T_i -T_{i-1})e^{-\sum_{k=1}^{i-1}X_k} \}
, 
\end{split}
\end{equation}
\color{black}
where $ c $ is a positive constant,  $ T_i $ is the occurrence time of the $ i $-th disaster and $X_{i}$
is the rate of the borrowers who survived the $ i $-th disaster, with the requirement $ X_{i}>0$ (it cannot be zero). Here we denote $\sum_{i=n}^m  \cdots = 0$ for $m <n$. 
The process can be understood as an
{\em exponential functional}
of the  renewal reward process
\begin{equation*}
    U_t = c \int_0^t
    e^{R_s} \,ds
\end{equation*}
where $R$
is the renewal-reward process associated 
with $ (T_i, X_i) $,
$ i=1, 2, \cdots $,
that is, 
\begin{equation*}
     R_t := - \sum_{j=1}^{J_t} X_j,
     \quad J_t = \sum_{k=1}^\infty
     1_{\{ T_i \leq  t \}}.
\end{equation*}
\begin{figure}[htbp]
\centering
\includegraphics[width = 5.5cm]{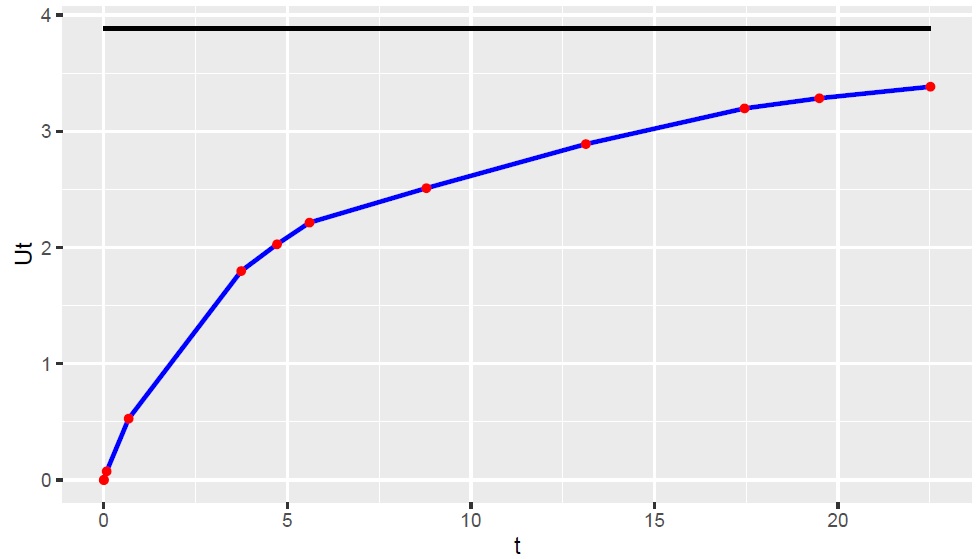}
\caption{A possible sample path.}
\label{fig.corina}
\end{figure}
Figure \ref{fig.corina} is the graphical representation of the repayment process. Each red dot represents a disaster hit. After every hit, the rate of repayment decreases since the affected mortgagors are forced to leave the pool. The repayment rate therefore slows down as disasters occur. Our aim is to make sure the path reaches the black line (the total amount lended $u$), in other words the full amount is repaid, within a finite time. Note that the model was inspired by the Japanese double-debt problem, but could account for any risk processes that temporises the effect of claims/losses.

Let $\tau_u $ be the first hitting time of $ u > 0 $, meaning the time the process $U$ (which starts at zero) reaches $u$ for the first time, which means the time the loan has been paid back in full. We shall study the probability of full loan repayment $\psi$, before a given time $t$,
\begin{equation*} 
\psi (u,t) := P(\tau_u < t),
\end{equation*}
and its corresponding probability of loan default in finite time horizon
\begin{equation*} 
\phi (u,t) := P(\tau_u  \geq  t) = 1 - \psi(u,t), 
\end{equation*}
for $u>0$ and $t >0$. \\

 {\bf The Proposal.}
%
The solution proposed in this paper is to include a small insurance premium in the contracted mortgage payments 
and in return the mortgagor would be relieved of the outstanding mortgage principal in the event of property destruction by a natural disaster. 
With every disaster hit, it is expected that properties will be destroyed and mortgagors will subsequently default. 
The premium will be set to ensure that the total amount of all {\it combined} principals is completely paid off within a finite time-scale. 
In our model, the premium rate for each mortgagor is going to be the same indifferent of the size of their loan or region of residence. 

The combined premiums from all of the mortgagors would be enough to cover the defaulted loans due to natural disasters. 
If a natural disaster occurs and causes damages of greater amount than the current reserve,  the remainder of damages costs would be paid off gradually as the premiums continue to be collected from the mortgagors not affected by the disaster. 
The insurance premium is calculated and set to ensure that the total amount of combined principals could be paid off in a finite time. \\

{\bf The Results.}
%
We derive the probability of full loan repayment,  $\psi$,
as a solution of 
an integral equation.
This expression could be further used for the numerical
evaluation of the default probability, 
but in this paper, 
for specific examples, 
we propose to use a calibration of the parameters of
the infinite-time full-loan-repayment-probability  
\begin{equation*}
\psi^{\infty} (u) := \lim_{t \to \infty} \psi (u,t) 
= P(\tau_u < \infty), 
\end{equation*}
and its counter-part, the probability of loan default
\begin{equation}\label{inftyruin}
\phi^{\infty} (u) := P(\tau_u  =  \infty) = 1 - \psi^{\infty}(u), 
\end{equation}
for all $ u > 0 $.
This strategy of approximation can be justified 
by Theorem \ref{Thm-convergence+} which shows that 
the difference, 
between the finite and infinite time probabilities, decays exponentially fast, as $ t \to \infty $, namely 
for any $u>0$, 
there exists $C(u)>0$ 
such that 
\begin{equation*}
0 \leq \psi^{\infty}(u) 
-\psi(u,t)\leq C(u) e^{-\xi  t},\ t \geq 0, 
\end{equation*} 
for some $\xi  >0$. 
%
The numerical experiments in 
section \ref{Nclexp} show that 
under a practical parameter set, the difference is negligible.

Furthermore, under the infinite horizon setting,  
as in \cite{Ragnar}, we will discuss the risk parameters for which the ruin probability stays within a range established by say a supervisory authority.
We will present the cases of Poisson, non-Poisson and randomized arrivals, all the while the ratio of clients remaining from the initial cohort is exponentially distributed. 
Concretely, we are considering  models with Poisson arrivals with parameter $\lambda$, followed by models with non-Poisson arrivals, as in Gamma$(2, \lambda)$ inter-arrival times, to account for some memory in the process. We study the probability of default or full repayment, by means of integro-differential equations and Laplace transforms, as in e.g. \cite{MR2586155, MR2965925, MR3033139}. Conjointly, we discuss the case of random Poisson parameter $\lambda$, equivalently to the distribution of inter-arrival time $W$ being an exponential random variable  with random parameter $\Lambda$,  to account for the clustering effect of the earthquake events. Thus one could estimate $\Lambda$ and  dynamically adjust the risks involved, as in \cite{Ragnar}.\\

{\bf Connections to exponential functionals of L\'evy processes.}
%
The default (or survival) probability in our  model can be interpreted as the distribution of an exponential functional of a stochastic process. Exponential functionals of L\'evy processes abound in both finance and insurance mathematics literature. 
In finance, the distribution of the exponential functional of a process is key in computing the Asian option price, and has attracted a lot of attention, featuring different approaches such as \cite{carmona1997distribution},
\cite{vecer2001pde},
and \cite{yor1992on}. Some further references on applications of exponentials of Brownian motion can be found in \cite{yor2001exponential}. In insurance, \cite{feng2019exponential} used the exponential functional of a L\'evy process in the analysis of an insurer liabilities when its variable annuity guarantees benefits on an exponential maturity of counterpart. 
%
%
%
Moreover, {\it perpetuities}, that can be seen as exponential functionals of a renewal-reward process, arise in a diverse range of fields, see \cite{vervaat_1979} 
and \cite{MR1311930} 
for a variety of references and examples of applications, 
including insurance \cite{MR1817589, MR2156557} 
and economics \cite{Dassios}. 
In \cite{MR1129194}, the distribution of a perpetuity presents applications in risk theory and pensions. 
\\

\color{black}
{\bf Connections to stochastic perpetuities.}
We note that our cash contribution process $ U $ is closely related to a
{\em stochastic perpetuity}, 
which is the present value of a stochastically discounted  series of independent, identically distributed (i.i.d.)
cashflows, that is, 
\begin{equation}\label{perp}
     D=\sum_{k\geq 1}
     \left(\prod_{i=1}^k d_i \right) 
     C_{k},
\end{equation}
where the cashflows,
$ (C_i) $, 
and the stochastic discount rates, $ (d_i) $, 
are mutually independent,  identically distributed
 sequences. 
The full-repayment probability in {\it infinite time} \eqref{inftyruin} 
is understood as the distribution function of 
a stochastic perpetuity. 
When considering cash flows $ C_k, $
arriving
at $ T_k $, 
$ k=1, 2 \cdots $,  and
discounted at a rate $ \delta $,
its present value
is given by 
$
    \sum_{k\geq 1}
    e^{-\delta T_k} C_k,
$
which is identified with $ D$ of \eqref{perp}
if we assume 
$ T_i -T_{i-1} $, 
$ i=1, 2, \cdots $
are i.i.d.,
by setting 
$d_i= e^{-\delta (T_i-T_{i-1})} $. 
In \cite{MR1129194} 
it is shown 
to exhibit a Gamma distribution for Poisson inter-arrivals of claims of exponentially distributed intensity (ex. 5.1.2.),
result that we can retrieve with our approach. 
Similar models are interpreted as risk models with stochastic returns on investments (see e.g. \cite{MR1817589, MR2156557, MR2538076}). 

The main difference between our model and a stochastic perpetuity (or a stochastic interest risk model) consists in the fact  that we are dealing with a stochastic process, whereas the perpetuity models are random variables. 
They are infinite sums of random variable, discounted (stochastic discounting) at fixed times, 
while our model is, 
at each time $ t $, 
a finite sum of random variables which can be seen as ``discounted at random times".\\

\color{black}
{\bf Connections to stochastic fixed point equations.}
The integral equation we derive for
the infinite horizon full repayment probability
can be seen as 
a {\em stochastic fixed point equation}.
As 
we will be looking at the tail, 
the structure of 
corresponding 
stochastic equation \eqref{eq:model} resonates with the {\it stochastic fixed point equations} described in \cite{MR1097468} 
for the first time, \cite{MR2538076} 
for insurance applications and amply analysed in the book of Mikosch \cite{Miko}. 
For such stochastic model of perpetuity-type, for various dependence structures, 
it is shown the tails of the distribution are regularly varying both in the univariate and multivariate cases. 
Here references \cite{kesten1973, MR1097468, MR3497380, MR2538076, MR2156557, MR1817589}. 
More recent literature 
could be found  in \cite{MIKOSCH2018, MR3981146}. 

Thus, in studying "default probabilities in infinite horizon", 
we are lead to stochastic fixed point equations, 
which are equal in distribution
to the ones for stochastic perpetuity.
Moreover, by choosing specific distributions for the random variables involved, the time "of discounting", or when an event occurs that would reduce the number of mortgagors, 
we can obtain explicit solutions of the tail distribution/ loan repayment. Stochastic fixed point equations literature abounds in asymptotic results. \\

\if
The main contribution  of our approach is that it permits an analysis of the finite time case, which would be the equivalent of the price of a stochastic annuity (not perpetuity),  
The distribution of a stochastic annuity is a non-trivial object, see discussion on stochastic life annuities in {\color{red} Dufresne (2007).}\\
\fi

 {\bf The Structure of the Paper.}
%
The rest of the paper is organized as follows.
In section \ref{sec2} we give 
a general equation for the loan default probability $\phi$ in a loan model driven by a general marked point process $(T_i,X_i). $ 
The finite time 
and infinite time 
default probabilities are expressed via Fredholm equations, with solutions to be analyzed asymptotically or numerically, if not readily available in closed forms.
We present an approximation of the finite-time probability from its infinite-time counterpart, in \ref{sec2-5-2}. 
In section \ref{sec3},
based on the infinite-horizon equation,
we consider the case of $X$ exponentially distributed and derive the ruin probability when $W$ are exponentially and 
Erlang distributed, respectively, under both independent and conditionally independent scenarios.
Finally, we calculate the ruin probability in the case of randomized arrival times when $W$ is exponentially distributed with random parameter $\Lambda$. One can then numerically calculate the insurance premium to be incorporated in the mortgage plan, such that the probability of default stays within a small range, as in Section \ref{numerical}. We conclude in Section \ref{conclusion}.\\

\noindent{\it Note}: In this paper, for keeping the context clear for the reader, the classical notation for ruin probability $\psi$ will be referred to as the {\it probability of loan repayment,}  or {\it probability of full loan repayment}, while the classical probability of non-ruin $\phi$ will be referred to as {\it probability of default}.

\section{Ruin probability in a general loan model}\label{sec2}
Earthquakes are considered rare, extreme events in most parts of the world, but could not be considered as such in Japan. 
The incidence of earthquakes has increased after the 2011 one. Looking at earthquake data, one can see that the arrivals of earthquakes can be described by one Poisson distribution before 2011, and a different parameter Poisson distribution after 2011. 
%
However, being still perceived as an extreme, rare event,  in addition to cultural reasons, at the moment only a minor part of home-owners in Japan have the relevant insurance cover. 
In order to evaluate the risks associated with a marketize earthquake insurance, we calculate the probability of default
of a lender, based on a model that can be seen as an {\it exponentiation of the Cram\'er-Lundberg model}, which is the classical model in collective non-life insurance (see e.g.  \cite{Dassios}, \cite{Miko},  \cite{Asmussen}).

Let  $W_i=T_i-T_{i-1}$, $ i=1,\ldots,$ $T_0=0.$  Firstly,
assume that
$ (W_i, X_i) $, $ i=1,\ldots,$ are 
independent and identically distributed with 
$(W,X)$.
For 
a bounded measurable function
$ h $, 
we set 
\begin{equation} \label{eq:k}
\begin{split}
\mathcal{K} h (u,t) := 
\mathbf{E}[I_{\{W <t\wedge \frac{u}{c}\}}h ( e^{X} (u-cW ),t-W)
]
.
\end{split}
\end{equation}
Then clearly
$ \mathcal{K} $
defines a linear transformation 
on $ L^\infty (\mathbf{R}_+^2 ) $.
\begin{theorem}\label{thm1}
The finite-time probability 
$ \phi $ satisfies the Fredholm type equation 
\begin{equation}\label{SPE-psi}
\begin{split}
\phi(u, t)= 
\phi_0(u, t)+  \mathcal{K} \phi (u, t), 
\end{split}
\end{equation}
where 
$
\phi_0(u,t )= \mathbf{P}( W >t )I_{\{t\leq \frac{u}{c}\}},
$ with $(u,t)\in (0,\infty)^2.$
\end{theorem}
\begin{proof}
Since $U$ is an increasing process, 
we have that 
$
\phi (u,t) 
= \mathbf{P}(  U_t \leq u )
. 
$
Since
\begin{equation*}\label{eq-1226-0}
\phi (u,t) 
=\mathbf{P}(U_t \leq u,\ T_1 \geq t)
+\mathbf{P}(U_t \leq u,\ T_1 <  t),
\end{equation*}
and since for $T_1 \geq t$, 
 $U_t = ct$, then
\begin{equation}\label{eq-1226-1}
\phi (u,t) 
=\phi_0 (u,t) 
+\mathbf{P}(U_t \leq u,\ T_1 <  t).
\end{equation}
When $T_1<t$, we have that 
\begin{equation*}
U_{t}= 
cT_1 + 
c e^{-X_1}\sum_{n=1}^\infty 1_{ [T_n , T_{n+1}) } (t) 
\{ (t-T_n) e^{-\sum_{i=2}^n X_i}
+ \sum_{i=2}^{n} (T_i -T_{i-1})e^{-\sum_{k=2}^{i-1}X_k} \}.
\end{equation*}
Noting that $(T_i, X_i)$ is a marked point process, 
\begin{equation*}
U_{t-T_1}'= c \sum_{n=1}^\infty 1_{ [T_n , T_{n+1}) } (t) 
\{ (t-T_n) e^{-\sum_{i=2}^n X_i}
+ \sum_{i=2}^{n} (T_i -T_{i-1})e^{-\sum_{k=2}^{i-1}X_k} \}
\end{equation*}
has the same distribution as $U_{t-T_1}$, and it's independent from $T_1$ and $X_1$. Consequently we have that 
\begin{equation}\label{eq-1226-2}
\begin{split}
\mathbf{P}(U_t \leq u,\ T_1 <  t) 
&= 
\mathbf{P}(cT_1 + 
 e^{-X_1}U_{t-T_1}' \leq u,\ T_1 < t) 
\\&= \mathcal{K} \phi (u,t).  
\end{split}
\end{equation}
Combining 
\eqref{eq-1226-1} with \eqref{eq-1226-2}, 
we obtain \eqref{eq-1226-0}. 
\end{proof}
\begin{corollary}
The finite-time probability of full loan repayment is expressed by a Fredholm type equation
\begin{equation}\label{eq:non-ruin}
\psi(u,t)
= \psi_0(u,t) + \mathcal{K} \psi(u,t), 
\end{equation}
where 
$\psi_0(u,t) := \mathbf{P} (W \geq \frac{u}{c}) I_{\{t > \frac{u}{c}\}},$  with  $(u,t)\in (0,\infty)^2.$
\end{corollary}
\begin{proof}
By theorem \ref{thm1}, for $\psi=1-\phi.$
\end{proof}
Let $\mathcal{A}_p$ be the (completion of the) set of all functions $h$ on $\mathbf{R}_+^2$ 
such that 
\begin{equation*}
||h||_{\mathcal{A}_p} := \sup_{t>0} \int_0^{\infty} |h(u,t)|^p\ du < \infty. 
\end{equation*}

\begin{lemma}\label{Lemma-L1}
Suppose that $ X $ is non-trivial.
Then $ \mathcal{K} $ defined by \eqref{eq:k} is
a linear operator on 
$\mathcal{A}_p$. Moreover, the operator norm 
$\displaystyle ||\mathcal{K}||_{\mathcal{L}(\mathcal{A}_p)} 
= \sup\{||\mathcal{K}h||_{\mathcal{A}_p} :\ ||h||_{\mathcal{A}_p} =1,\ h \in \mathcal{A}_p \} $ 
is strictly less than $1$,
for any $ 1 \leq p < \infty$. 
\end{lemma}
\begin{proof}
For $h \in \mathcal{A}_p$, 
\if2 
we see that 
\begin{equation*}
\begin{split}
||\mathcal{K}h||_{\mathcal{A}_p} 
= 
\sup_{t>0} \int_0^{\infty} |\mathbf{E}[I_{\{W <t \wedge \frac{u}{c}\}}
h( e^{X} (u-cW ),t-W)]|^p\ du.
\end{split}
\end{equation*}
\fi 
by using Jensen's inequality and Minkowski's inequalities, we have that 
\begin{equation*}
\begin{split}
||\mathcal{K}h||_{\mathcal{A}_p}  
&\leq \mathbf{E}[\sup_{t>0} \int_0^{\infty} |I_{\{W <t \wedge \frac{u}{c}\}}
h( e^{X} (u-cW ),t-W)|^p\ du]
\\& 
\leq   \mathbf{E}[
e^{-X}\sup_{t>0}\int_{0}^{\infty} |h( Z,t)|^p\ dZ]
=
||h ||_{\mathcal{A}_p}
\mathbf{E}[
e^{-X}].
\end{split}
\end{equation*} 
Since $X$ is non-trivial, we completed the proof. 
\end{proof}

\begin{lemma}\label{Lemma-L2}
Suppose that the first moment of $ W$ 
is finite. Then $\psi_0 \in \cap_{p \geq 1} \mathcal{A}_p 
$.
\end{lemma}
\begin{proof}

Let $p\geq 1$. We note that 
\begin{equation*}
\begin{split}
||\psi_0||_{\mathcal{A}^p}
= \sup_{t >0} 
\int_0^{\infty}\psi_0(u,t)^p \, du  \, dt 
= \sup_{t >0}
\int_0^{ct}
\left(\mathbf{P} (W \geq \frac{u}{c})\right)^p \, du. 
\end{split}
\end{equation*}
Having an increasing function,
\begin{equation*}
\begin{split}
\sup_{t >0}
\int_0^{ct}
\left(
\mathbf{P} (W \geq \frac{u}{c} )\right)^p \, du
=
\int_0^{\infty}
\left(\mathbf{P} (W \geq \frac{u}{c})\right)^p \, du
\leq
\int_0^{\infty}
\mathbf{P} (W \geq \frac{u}{c}) \, du
=c\mathbf{E}[W],
\end{split}
\end{equation*}
with the inequality given by the fact that the probability ranges in [0,1].
Hence we conclude $\psi_0 \in \mathcal{A}_p,$
by the boundedness of the first moment of $W$. 
\end{proof}

\begin{theorem}\label{theorem-25}
Under the assumptions of Lemmas \ref{Lemma-L1} and \ref{Lemma-L2}, 
we have that 
\begin{equation*}
\psi = \sum_{n=0}^{\infty} 
\mathcal{K}^n \psi_0 
\in L( \cap_{p \geq 1}\mathcal{A}_p),\ 
\end{equation*}
and moreover
$\phi = 1- \sum_{n=0}^{\infty} 
\mathcal{K}^n \psi_0.$
\end{theorem}
\begin{proof}
Since  
the operator norm 
of $ \mathcal{K} $ is strictly less than $ 1, $ by Lemma \ref{Lemma-L1}, 
the Neumann series
$ \sum_{n=0}^\infty 
\mathcal{K}^n $
is convergent in $ L (\cap_{p \geq 1}\mathcal{A}_p) $
and thus defines (describes)
the inverse operator
of $ 1 - \mathcal{K}$.
\end{proof}

\begin{remark}
All the results in this section are valid for possibly negative $ X $, as long as $ \mathbf{E} [e^{-X}] < 1 $.
Although such situations are not realistic in our mortgage loan modelling, it  might be applicable to different contexts.
Furthermore, since the results cover the exponential functionals of compound Poisson processes, this may contribute to the literature on exponential functionals of L\'evy processes 
(see e.g. \cite{Bertoin-Yor}).
\end{remark}

\section{From finite-time to infinite-time ruin probability}\label{sec2-5-2}
In this section we will work on infinite horizon ruin probabilities as limit of finite-time ruin probabilities when 
the time goes to infinity. By continuity of probability measure, it holds that 
\begin{equation*}
\psi^{\infty}(u) = \lim_{t \rightarrow \infty} \psi (u,t) = \mathbf{P}(\tau_u < \infty), 
 u>0. 
\end{equation*} 
We assume that the assumptions of Lemmas \ref{Lemma-L1} and \ref{Lemma-L2} are satisfied, 
that is, $\mathbf{P} (X>0) >0$ and $\mathbf{E} [W] < \infty$. 
Moreover, we assume that $\mathbf{P} (W>0) >0$ and that the joint density function of $W$ and $X$ exists, and we denote it by $f_{W,X}(w,x)$.
Let  $L_p $ be the collection of functions which satisfy 
$\int_0^{\infty} |g(u) |^p \, du < \infty, $ 
for $g\in L_p$. 
By identifying functions $g \in L_p,  $  via $g \in \mathcal{A}_p$ with $g  (u,t ) =g(u ) \ (u,t >0)$, 
$L_p$ is a subset of $\mathcal{A}_p$.  
We define an operator $\mathcal{K}_{\infty}$ on $ L_p$ by  
\begin{equation*}\label{def-op-K+}
\begin{split}
\mathcal{K}_\infty g (u) := 
\mathbf{E}\left[1_{\{cW  < u\}}
g (e^{X} (u -cW) ) 
\right],  \quad u>0
.
\end{split}
\end{equation*}
Clearly,  
   $\mathcal{K}_{\infty}  g (u) 
    = \lim_{t \to \infty } 
    (\mathcal{K} g) (u,t),$
%
and therefore, by Lemma \ref{Lemma-L1}, the operator norm of $\mathcal{K}_{\infty}$ is strictly less than one, and it 
leads to
\begin{equation}\label{psi_inf_dec}
\psi^{\infty}  = \sum_{n=0}^{\infty} 
\mathcal{K}^n_{\infty} \psi_0^{\infty} 
\in  \cap_{p  \geq 1}L_p,\ 
\end{equation}
where $\psi_0^{\infty} (u) := \mathbf{P} (W> \frac{u}{c})$. 
We note that for any $t>\frac{u}{c}$, $\psi_0(u,t)$ is equal to $\psi_0^{\infty} (u)$. 

\subsection{Convergence rate of ruin probability from finite-horizon to infinite-horizon}\label{sec2-5}
The finite-time horizon probability of default is extremely relevant for the mortgage markets that deals fixed-term loans. Having such a fast convergence to the infinite-time probability, we can actually use the infinite-time probability (which we can more often calculate explicitly) as an approximation for the finite-time one.
In this section,
we show that the finite-time ruin probability converges to the
infinite-horizon one exponentially  
fast. 
\begin{theorem}\label{Thm-convergence+}
We assume that there exists $\theta_0 >0$ such that the moment generating function 
$M_W(\theta)=\mathbf{E}[e^{\theta W} ] $ exists for any $ \theta < \theta_0 $. 
Then, we have the following estimates.
\begin{enumerate}
\item[(i)] 
There exists $\xi>0$ and $C > 0 $ such that for any $t >0$, 
\begin{equation}\label{thm-exp-eq2}
\begin{split}
&\int_0^\infty  |\psi^{\infty}(u) 
-\psi(u,t)| 
\, du  \leq C e^{-\xi t} 
.\end{split}
\end{equation}
\item[(ii)] For any $u>0$, there exists $C(u) > 0 $ such that 
for any $t>0$,
\begin{equation*}
0 \leq \psi^{\infty}(u) 
-\psi(u,t)\leq C(u) e^{-\xi' t} 
,
\end{equation*}
\end{enumerate} 
for some $\xi' >0$.
\end{theorem}
\begin{proof}
We first show that 
(i) implies (ii). 
By (i), we can choose
$\xi >0 $ and $C>0$ such that 
\eqref{thm-exp-eq2} holds true. 
Let $\xi' \in ( 0,\xi )  $.   
For each $ n \in \mathbf{N
}$, set 
\begin{equation*}
A_n := \{u >0 ;\ \psi^{\infty}(u) 
-\psi(u,n)\geq C e^{- \xi'  n} \}. 
\end{equation*} 
Then 
{the Lebesgue measure of the event $A_n$ is bounded} {as}
\begin{equation*}
\begin{split}
\mbox{Leb}( A_n)  
 \leq \frac{1}{C e^{- \xi'n}} \int_0^{\infty }|\psi^{\infty}(u) -\psi(u,t) |   \, du \leq \frac{Ce^{- \xi n}}{ C e^{- \xi'  n}} 
=e^{-(\xi-\xi') n}
. \end{split}
\end{equation*}
Therefore we have that 
$\sum_{n=0}^{\infty} \mbox{Leb}( A_n)  
< \infty.
$
Hence, by Borel-Cantelli lemma, 
$\mbox{Leb} (\limsup_{n \rightarrow \infty } A_n)  =  0, $
which leads to 
\begin{equation*}
\psi^{\infty}(u) 
-\psi(u,t) < C(u) e^{- \xi'  t},\ \mbox{Leb-a.s. }u, 
\end{equation*}
for some $C(u)>0$. 
Since $\psi^{\infty}(u)$ and  $\psi(u,t)$ are continuous in $u$, it holds for any $ u $. 

Next we prove (i). 
Since 
\begin{equation*}
\psi^{\infty}(u) 
-\psi(u,t)=
\sum_{n=0}^{\infty}
(\mathcal{K}_{\infty}^n\psi_{0}^{\infty}(u) 
-\mathcal{K}^n\psi_0(u,t))
,
\end{equation*} 
by Theorem \ref{theorem-25} and \eqref{psi_inf_dec},
it suffices to show that 
there exist constants 
$ C'>0$, $ a \in (0,1) $, 
and $ \xi > 0 $ such that 
\begin{equation}\label{target}
    \int_0^\infty 
    | \mathcal{K}_{\infty}^n\psi_{0}^{\infty}(u) 
-\mathcal{K}^n\psi_0(u,t)| du 
\leq C' a^n e^{-\xi t}. 
\end{equation}

To establish \eqref{target}, 
we rely on the following recursive {relation}:
\begin{equation}\label{thm-l1-Kn-rec} 
\begin{split}
&\int_0^\infty  |
\mathcal{K}_{\infty}^n\psi_{0}^{\infty}(u) 
-\mathcal{K}^n\psi_0(u,t)| 
\, du 
\\& \leq
\mathbf{E}[e^{-X}]
\mathbf{E}[I_{\{W<t\}}\int_0^{\infty}
 | \mathcal{K}_{\infty}^{n-1}\psi_{0}^{\infty}(u)  - \mathcal{K}^{n-1}\psi_{0}  (u, t-W) |\, du]
\\& \qquad 
+ c 
\mathbf{E}[e^{-X}]^n \mathbf{P}(W > t) 
\mathbf{E}[ W ] 
. \end{split}
\end{equation}
A proof of \eqref{thm-l1-Kn-rec} will appear in Appendix.  

{Using \eqref{thm-l1-Kn-rec} recursively}, 
\begin{equation}\label{thm-n-eq5-26} 
\begin{split}
&\int_0^\infty  |
\mathcal{K}_{\infty}^n\psi_{0}^{\infty}(u) 
-\mathcal{K}^n\psi_0(u,t)| 
\, du 
\\& \leq 
\mathbf{E}[e^{-X}]
    \mathbf{E}[I_{\{W_1<t\}}  \\& \quad 
    \times\Big( \mathbf{E}[e^{-X}] 
\mathbf{E} [I_{\{W_2<t -W_1\}} 
\int_0^{\infty}
 | \mathcal{K}_{\infty}^{n-2}\psi_{0}^{\infty}(u)  - \mathcal{K}^{n-2}\psi_{0}  (u, t-W_1 -W_2 ) |\, du | W_1] ]
\\& \quad 
+ c
\mathbf{E}[ e^{-X}]^{n-1} \mathbf{E}[ W ] 
\mathbf{E}[I_{\{W_1<t\}}
 \mathbf{P}(W_2 > t -W_1 | W_1)  
] \Big) 
+c 
\mathbf{E}[e^{-X}]^n \mathbf{P}(W_1 > t) 
\mathbf{E}[ W ] 
\\& = 
\mathbf{E}[e^{-X}]^2  
\mathbf{E}[I_{\{S_2<t\}}
\int_0^{\infty}
 | \mathcal{K}_{\infty}^{n-2}\psi_{0}^{\infty}(u)  - \mathcal{K}^{n-2}\psi_{0}  (u, t-S_2 ) |\, du ]
\\& \qquad  
+c
\mathbf{E}[e^{-X}]^n \mathbf{P}(S_2 > t) 
\mathbf{E}[ W ] 
\\& \leq 
\mathbf{E}[e^{-X}]^n
\mathbf{E}[I_{\{S_n<t\}}
\int_0^{\infty}
 |\psi_{0}^{\infty}(u)  -\psi_{0}  (u, t-S_n ) |\, du ]
+c 
\mathbf{E}[e^{-X}]^n \mathbf{E}[ W ]
\mathbf{P}(S_n > t) 
 , \end{split}
\end{equation}
where $S_n= \sum_{k=1}^n W_k $. 
The $L^1$- norm of the difference between the finite and the infinite ruin probabilities 
in the right-most of \eqref{thm-n-eq5-26} is 
calculated as 
\begin{equation}\label{eq-L1h0-26}
\begin{split}
\int_0^{\infty}|\psi_{0}^{\infty}(u) - \psi_0(u,t) |\, du  
 = \int_{ct}^{\infty} \mathbf{P}(W >\frac{u}{c})  \, du  
 = c\mathbf{E}[I_{\{W >t\}}(W-t)],\
\end{split}
\end{equation}
for $t >0 $. 
Therefore by substituting \eqref{eq-L1h0-26} to \eqref{thm-n-eq5-26},  we have that 
\begin{equation}\label{thm-n-eq6-26} 
\begin{split}
&\mbox{(the right hand side of \eqref{thm-n-eq5-26} ) }
\\& = c \mathbf{E}[e^{-X}]^n
\mathbf{E}[I_{\{S_n<t\}}\mathbf{E}[I_{\{W_{n+1} > t- S_n\}}(W_{n+1} - (t- S_n) ) |\ S_n ]  ]
\\& \qquad  
+c 
\mathbf{E}[e^{-X}]^n \mathbf{E}[ W ]
\mathbf{P}(S_n > t) 
\\& = c 
\mathbf{E}[e^{-X}]^n
\mathbf{E}[I_{\{  S_{n+1} > t\}}(W_{n+1} - ((t- S_n)\vee 0)  )   ]
. \end{split}
\end{equation}
The claim \eqref{target} 
is fulfilled once 
we establish 
\begin{equation}\label{target2}
    \mathbf{E}[I_{\{  S_{n+1} > t\}}(W_{n+1} - ((t- S_n)\vee 0))] 
    \leq e^{-\xi t} C'' (a')^n
\end{equation}
with some $ \xi, C'' > 0 $
and $ a'< 1/\mathbf{E}[e^{-X}]$. 
To see this, we note that 
for any $\theta \in (0, \theta_0)  $,  
\begin{equation}\label{thm-n-eq7-26} 
\begin{split}
\mathbf{E}[I_{\{  S_{n+1} > t\}}(W_{n+1} - ((t- S_n)\vee 0)  )   ]
 \leq \mathbf{E}[I_{\{  S_{n+1} > t\}}W_{n+1}]
 \leq \mathbf{E}[W_{n+1}e^{\theta ( S_{n+1} -t)  }],
\end{split}
\end{equation}
allowing for infinite value in the rightmost.
Since $W_{n+1}$ and $S_n$ are independent, the right hand side of \eqref{thm-n-eq7-26} 
is decomposed as follows: 
\begin{equation*} 
\begin{split}
\mathbf{E}[W_{n+1}e^{\theta ( W_{n+1}  + S_{n} -t)  }]
 = e^{-\theta t}\mathbf{E}[W_{n+1}e^{\theta W_{n+1}}]\mathbf{E}[e^{\theta S_n}]
 = e^{-\theta t} 
\mathbf{E}[We^{\theta W}]
M_W^n (\theta) 
. \end{split}
\end{equation*}
Since $X$ is non-trivially positive, we can choose $\xi  \in (0, \theta_0) $ such that 
\begin{equation*}
M_W(\xi) :=
\mathbf{E} [e^{\xi W} ]  <  \frac{1}{\mathbf{E} [e^{-X} ]}.
\end{equation*} 
Thus, we established 
\eqref{target2},
which implies \eqref{target},
and hence the proof is complete. 
\end{proof}

\subsection{Exponential decay rate for exponential time arrival}

Firstly, we
consider that 
the arrivals of disasters $W$ are exponentially distributed, with parameter $\lambda$, {meaning the moment generating function exists for all $\theta < \lambda$, $M_W(\theta) = \frac{\lambda}{\lambda -\theta}$, so Theorem \ref{Thm-convergence+} applies, moreover the constants $c$ and $C$ of \eqref{thm-exp-eq2} can be determined explicitly. 

\begin{proposition}
For $W$ exponentially distributed with parameter $\lambda>0,$
$$\int_0^\infty  |\psi^{\infty}(u) 
-\psi(u,t)| \, du \leq \frac{1}{\alpha}  e^{- \lambda \alpha t}, \quad t\geq0, 
$$
where $\alpha = 1- \mathbf{E}[e^{-X} ].$
\end{proposition}}
\begin{proof}
The partial sum $S_n = \sum_{k=1}^n W_k $ 
is Erlang distributed with parameters {$n$ and} $\lambda$, and its tail distribution function is given by 
\begin{equation*}\label{eq-tail-Sn}
 \mathbf{P}(S_n > z) 
= e^{- \lambda z } \sum_{k=0}^{n-1} \frac{(\lambda z)^k} {k! }. 
\end{equation*}
For $n \geq 0$, 
by tower property of expectations, conditioned by $S_n$, we have 
that 
\begin{equation}\label{eq-ex1} 
\begin{split}
\mathbf{E}[I_{\{  S_{n+1} > t\}}(W_{n+1} - ((t- S_n)\vee 0)  )   ]
&= \mathbf{E}[\int_{t - S_n}^{\infty}(w - ((t- S_n)\vee 0)  ) \lambda e^{-\lambda w} \, dw    ]
\\&= \mathbf{E}[I_{\{S_n < t\}}e^{ - \lambda (t- S_n)} ] + \mathbf{P}(S_n>t)
. \end{split}
\end{equation}
By \eqref{eq-tail-Sn}, the first term of \eqref{eq-ex1} is rewritten as 
\begin{equation}\label{eq-ex2} 
\begin{split}
 e^{ - \lambda t }\mathbf{E}[I_{\{S_n < t\}}e^{ \lambda S_n} ]
=e^{ - \lambda t } \frac{(\lambda t)^n} {n! }
. \end{split}
\end{equation}
By the combination of \eqref{thm-n-eq6-26}, \eqref{eq-ex1} and  \eqref{eq-ex2},   
 we obtain that 
\begin{equation}\label{eq-ex3} 
\begin{split}
&\int_0^{\infty} | \mathcal{K}_{\infty}^n\psi_{0}^{\infty}(u) 
-\mathcal{K}^n\psi_0(u,t) | \, du 
 \leq 
\mathbf{E}[e^{-X} ]^n  e^{- \lambda t } \sum_{k=0}^{n} \frac{(\lambda t)^k} {k! }, \quad {\forall n\geq 0}
. \end{split}
\end{equation}
Summing up \eqref{eq-ex3} for $n \geq 0$, 
we get the decay rate 
\begin{equation*}\label{eq-ex4} 
\begin{split}
&\int_0^\infty  |\psi^{\infty}(u) 
-\psi(u,t)| 
\, du   
\leq \sum_{n=0}^{\infty}
\mathbf{E}[e^{-X} ]^n  e^{- \lambda t } \sum_{k=0}^{n} \frac{(\lambda t)^k} {k! } 
=  \frac{\exp(- \lambda t(1- \mathbf{E}[e^{-X} ] ) )}{1 - \mathbf{E}[e^{-X} ] } 
. \end{split}
\end{equation*}
\end{proof}

\section{Explicit ruin probabilities in loan models}\label{sec3}

{We will set the distribution parameters for earthquake frequency and severity 
to values that minimize the probability of a cohort not repaying the loans, that we will call ruin/default. 
In real life, historical data should be used to determine these parameters and further set the insurance premium which minimizes this probability of ruin.}
We will
work on 
infinite horizon {default/full repayment} probabilities, 
for which 
the equation
\eqref{SPE-psi}
reduces to 
\begin{equation}\label{SPE-phi-infty}
\begin{split}
\phi^{\infty}(u)= 
\phi^{\infty}_0 (u) +  
\mathcal{K}_\infty \phi^{\infty} (u) ,
\end{split}
\end{equation}
with $\phi^{\infty}_0 (u)= \lim_{t \to \infty}\mathbf{P}( W >t )I_{\{t\leq \frac{u}{c}\}}=0$ and the equation 
for $\psi^{\infty} $ is 
\begin{equation}\label{SPE-psi-infty}
\begin{split}
\psi^{\infty} (u) = \psi^{\infty}_0 (u) + \mathcal{K}_\infty \psi^{\infty} (u),
\end{split}
\end{equation}
where $\psi^{\infty}_0 (u)=  \lim_{t \to \infty}\mathbf{P} (W \geq \frac{u}{c}) I_{\{t > \frac{u}{c}\}} = \mathbf{P}(W  \geq  \frac{u}{c})$. \\

In the following, we 
consider that the proportion $e^{-X}$ of borrowers left in the loan programme is described by an exponentially distributed random variable $X,$ with parameter $
\alpha$. {In the remainder, for ease of notation, we will use $\phi(u)$ for $\phi^{\infty}(u)$ and the equivalent for $\psi.$}

\color{black}

\subsection{Memory-less Arrivals}

We will start with Poisson arrivals, meaning that we wait an exponential amount of time between events. We also assume that the arrivals of disasters $W$ are independent of the effect on the bank's remaining proportion of borrowers $X$ .
\begin{proposition}\label{ex1}
If $X$ is exponentially distributed with parameter $\alpha >0$ and $W$ is exponentially distributed with parameter $\lambda >0$, and independent of $X$,
then the probability of default $\phi$ is given by 
\begin{equation}\label{prop4-phi}
\begin{split}
\phi(u) 
&=  \frac{1}{\Gamma (\alpha +1)} 
\int_{0 }^{\frac{\lambda u}{c}}  z ^{\alpha}\exp ( -z) dz, \quad u\geq0.
\end{split}
\end{equation}
\end{proposition}
\begin{proof}
Here we give a direct proof using only differentiation.
From 
\eqref{SPE-phi-infty} and using the joint density of $(W, X),$ 
$f_{W,X} (w,x)=  \lambda \alpha e^{-\lambda w -\alpha x}$,
we have that the probability of default $\phi$ satisfies
\begin{equation}\label{phi++}
\begin{split}
 \phi(u) 
&= 
\int_0^{\infty} \int_{0}^{\frac{u}{c}}
\phi (e^{x} (u -cw) ) 
\lambda \alpha e^{-\lambda w -\alpha x} \, dw\, dx
.
\end{split}
\end{equation}
By differentiating both sides of \eqref{phi++}, twice, we obtain the following ordinary differential equation with non-constant coefficients
\begin{equation*}
\begin{split}
\phi''(u)
= \left(-\frac{\lambda }{c} +\frac{\alpha}{u}\right) \phi' (u),
\end{split}
\end{equation*}
equipped with two boundary conditions, one at infinity and one at zero.
Using the boundary condition $\phi(0)=0$, the solution of the differential equation reads
\begin{equation*}
\begin{split}
\phi(u) 
 = C \left(\frac{c}{\lambda}\right)^{\alpha +1}  
\int_0^{\frac{\lambda u}{c}}  z ^{ \alpha}\exp ( -z) dz. 
\end{split}
\end{equation*}
Furthermore, from
the infinity condition
$\lim_{u \to \infty }
\phi(u) = 1,$
$C = \frac{1}{\Gamma (\alpha  +1)} \left(\frac{\lambda}{c} \right)^{\alpha +1},$ completing the proof.
\end{proof}

\subsection{General Arrivals}
\begin{proposition}\label{prop-hatphi}
If $X$ is exponentially distributed with parameter $\alpha  >0$ and $W$ is a random variable independent of $X$, 
with density a positive integrable
function $ f_W $, then 
the Laplace transform of the ruin probability 
satisfies 
\begin{equation}\label{2-phi12}
\begin{split}
\hat{\psi} (s) 
&= s^{-1} 
(1 - \hat{f}_W (cs))
+ 
\alpha s^{\alpha-1} 
\hat{f}_W(cs)
 \int_0^s u^{\alpha -1} (1 - \hat{f}_W (cu)) e^{\int_u^s\alpha v^{-1}\hat{f}_W(cv)\, dv } \, du,
\end{split}
\end{equation}
where  $\hat{f}$ denotes the Laplace transform of $f$. 
\end{proposition}
\begin{proof}
The joint density function of $(W, X) $ is
$f_{W,X} (w,x)= \alpha f_W(w)e^{-\alpha x}$.
By taking Laplace transform for the both sides of \eqref{SPE-psi-infty}, we have that,
for $s>0$, 
\begin{equation*}\label{label-prop4-1-0}
\begin{split}
\hat{\psi} (s)
=  \frac{1}{s} (1 - \hat{f}_W (cs)) 
+ \int^{\infty}_{0} e^{-su}
 \int^{\infty}_{0}
\int_0^{\frac{u}{c}}
\psi (e^{x}(u-cw)) 
\alpha f_W(w)e^{-\alpha x}\, dw\, dx \, du. 
\end{split}
\end{equation*}
By changing the order of
the integrals of the second term, which is possible since they are all positive, 
we obtain that 
\if1 
\begin{equation*}
\begin{split} 
\hat{\psi} (s) 
&=\frac{1}{s} (1 - \hat{f}_W (cs))
+
\alpha s^{-(\alpha+1) }
\hat{f}_W(cs)
\int^{s}_{0} 
x^{\alpha}
\hat{\psi} (x) 
\, dx 
.
\end{split}
\end{equation*}
This can be re-written
\fi 
\begin{equation*}
s^{\alpha}\hat{\psi} (s) 
= s^{\alpha -1} 
(1 - \hat{f}_W (cs))
+ 
\alpha s^{-1} 
\hat{f}_W(cs)
\int^{s}_{0} 
x^{\alpha}
\hat{\psi} (x) 
\, dx 
.\end{equation*}
By denoting 
$G(s)=\int^{s}_{0} x^{\alpha}\hat{\psi} (x) \, dx, $
we have a separable ordinary differentiable equation in $G(s)$,
\begin{equation*}\label{eq-prop32}
G'(s)= 
s^{\alpha -1} (1 - \hat{f}_W (cs))
+ 
\alpha s^{-1}\hat{f}_W(cs)G(s), 
\end{equation*}
which leads to 
\begin{equation*}\label{eq-prop32-2}
G(s)= 
\left( G(t) + \int_t^s u^{\alpha -1} (1 - \hat{f}_W (cu)) e^{-\int_t^u\alpha v^{-1}\hat{f}_W(cv)\, dv } \, du \right)
e^{\int_t^s\alpha v^{-1}\hat{f}_W(cv)\, dv }, 
\end{equation*}
for $0 <t <s$. Since $\lim_{t \rightarrow 0} G(t) = 0$, 
by taking $t \rightarrow 0$ for the both sides, 
we obtain that  
\begin{equation*}\label{eq-prop32-3}
G(s)= 
 \int_0^s u^{\alpha -1} (1 - \hat{f}_W (cu)) e^{\int_u^s\alpha v^{-1}\hat{f}_W(cv)\, dv } \, du, 
\end{equation*}
and thus we conclude that
\eqref{2-phi12} is verified. 
\end{proof}

\begin{remark}

When $W$ is exponential with parameter $\lambda >0$,
\begin{equation*}
\hat{F}_W(cs) =\int_0^{\infty} e^{-cs x} e^{-\lambda x} dx= \frac{1}{\lambda+cs}
\end{equation*}
and thus
\begin{equation*}
\begin{split}
\hat{\psi} (s) 
 &= \frac{1}{s} - \frac{c}{\lambda+cs} + c\alpha s^{\alpha} (\lambda+cs)^{\alpha -1} \int_0^s \left(\frac{u}{\lambda +cu}\right)^{\alpha }  \left(\frac{1}{u} - \frac{c}{\lambda +cu} \right) \, du,
\end{split} 
\end{equation*}
which after Laplace inversion leads to the results of Proposition \ref{ex1}.

\end{remark}
\subsection{Arrivals with Memory}

When the inter-arrival times
$ W $ are Erlang distributed, with
shape parameter $k$
and rate parameter 
$ \lambda $, we can obtain again explicit results for the ruing probability.
Recall the density function of an Erlang distribution is given by
\begin{equation*}
f(w;k, \lambda) 
= \frac{\lambda^k 
w^{k-1} e^{-\lambda w}}{(k-1)!}, \, x >0,\ 
k \in \mathbf{N}, \lambda >0,\ w >0. 
\end{equation*}
\begin{proposition}\label{ex2}
If $X$ is exponentially distributed with parameter $\alpha > 0$ and $W$ is Erlang  distributed
with shape parameter $2$ and rate parameter $\lambda$, then $\phi$ is given by 
\begin{equation*}\label{phi2}
\begin{split}
\phi(u) 
&= C g(u), \quad u\geq 0
\end{split}
\end{equation*}
where 
\if1
\begin{equation*}
g(u) = \int_0^{\frac{\lambda u}{c}} 
\frac{1}{\sqrt{ y}}
e^{-y}
\cosh \left(2 \sqrt{\alpha  y }\right)
\left(\int_0^
{\frac{\lambda }{c} -y } 
e^{-x}
x^{\alpha -\frac{3}{2}}\, dx\right)
\, dy
\end{equation*}
\fi
\begin{equation*}
g(u) = \int_0^{\frac{\lambda u}{c}} 
\frac{1}{\sqrt{ y}}
e^{-y}
\cosh \left(2 \sqrt{(\alpha +2)   y }\right)
\left(\int_0^
{\frac{\lambda u}{c} -y } 
e^{-x}
x^{\alpha +\frac{1}{2}}\, dx\right) dy, \quad C= \lim_{u \rightarrow \infty}\frac{1}{g(u)}.
\end{equation*}
\end{proposition}
\begin{proof}
Here $f_W(w)$ is $f(w;2,\lambda)$. The Laplace transform of $f_W$ is given by 
\begin{equation*}
\begin{split}
\hat{f}_W (s) &= \int_0^{\infty} 
e^{-sw}f_W(w)
\, dw = \frac{\lambda^2}{(s+ \lambda )^2}, 
\end{split} 
\end{equation*}
and its derivative is 
\begin{equation*}
\begin{split}
(\hat{f}_W)' (s) &=  -\frac{2\lambda^2}{(s+ \lambda )^3}. 
\end{split} 
\end{equation*}
\color{black}
By Proposition \ref{prop-hatphi}, we have that,
for some constant $C$, 
\begin{equation*}\label{prop6-1}
\begin{split}
\hat{\phi}(s)
&= C F(s) 
, \quad s>0,
\end{split}
\end{equation*}
where 
$
F(s)= 
\frac{1}{ s
\left(s + \frac{\lambda}{c} \right)^{2+\alpha}}
 e^{
\frac{\alpha  \lambda}{cs+\lambda}
}  \quad s> 0.
$
We claim that
\begin{equation}\label{prop6-in-F}
\begin{split}
&\mathcal{L}^{-1}
(F)(t)
\\&= 
\frac{1}{\Gamma (\alpha +\frac{3}{2})\sqrt{\pi }}
(\frac{c}{\lambda })^{\alpha+2 }
\int_0^{\frac{\lambda t}{c}} 
\frac{1}{\sqrt{ y}}
e^{-y}
\cosh \left(2 \sqrt{(\alpha +2)   y }\right)
\left(\int_0^
{\frac{\lambda t}{c} -y } 
e^{-x}
x^{\alpha +\frac{1}{2}}\, dx\right)
\, dy.
\end{split}
\end{equation} 
This can be seen 
in the following way. 
We set 
 $F_1 (s) = s^{-1}(s +\frac{\lambda}{c}) ^{-\alpha -\frac{3}{2} }$ 
and 
$F_2 (s) = s^{-\frac{1}{2}}
e^{\frac{(\alpha +2)\lambda }{cs}}$. 
Then the inverse Laplace transform of $F_1$ and 
$F_2$ are given by  
\begin{equation*}\label{prop6-in-F1}
\mathcal{L}^{-1}
(F_1)
(t)=
\frac{1}{\Gamma (\alpha +\frac{3}{2})}
(\frac{c}{\lambda })^{\alpha +\frac{3}{2}}
\int_0^{\frac{ \lambda t}{c}} 
e^{-x}
x^{\alpha +\frac{1}{2}}\, dx, 
\end{equation*}
and 
\begin{equation*}
\mathcal{L}^{-1}
(F_2)
(t)=\frac{1}{\sqrt{\pi t}}
\cosh \left(2 \sqrt{\frac{(\alpha +2) \lambda t}{c}}\right),
\end{equation*}
(see e.g. Chapter 5.5 (32) in \cite
{MR0061695}). 
Therefore we obtain \eqref{prop6-in-F}.
Hence we conclude that 
\begin{equation*}\label{prop6-1-fin}
\phi(s)
= C \mathcal{L}^{-1}(F)(s).
\end{equation*}
{where $\mathcal{L}^{-1}(F)$ is our $g$}.
Since $\lim_{u \rightarrow \infty }\phi(u) =1$, 
$C= 
\lim_{u \to \infty}\frac{1}{g(u)}$. 
\end{proof}
\color{black}

\subsection{Randomized Arrival Times}\label{random-lambda}
  We assume that 
$W$ is exponentially distributed with random parameter $\Lambda.$ For any given $\Lambda=\lambda$, $W$ is exponentially distributed with random parameter $\lambda.$ This special dependence structure is referred to as conditional independence.
As in \cite{MR2799308}, for every given $\Lambda=\lambda$, we calculate the ruin probability, say $\psi_{\lambda}(u),$ and then integrate over all the possible values of $\Lambda$, with distribution function $F_{\Lambda}$, leading to
$$\psi(u)=\int_0^{\infty} \psi_{\lambda}(u) d F_{\Lambda}(\lambda).$$
From
\eqref{prop4-phi}, 
\begin{equation}\label{20180316-1}
\begin{split}
P (\tau_u = \infty | \Lambda = \lambda )
= :
\phi_{\lambda}(u) 
&:=  \frac{1}{\Gamma (\alpha +1 )} 
\int_{0 }^{\frac{\lambda u}{c}}  z ^{\alpha}\exp ( -z)\, dz,
\\
\end{split}
\end{equation}
for each $\lambda (>0)$, 
and then
the probability of loan repayment satisfies 
\begin{equation*}
\begin{split}
\phi(u)&=\int_0^{\infty} \phi_{\lambda}(u) d F_{\Lambda}(\lambda),
\end{split}
\end{equation*}
where $ F_{\Lambda}(\lambda)$ is the distribution function of $\Lambda$. 
When $\Lambda$ is Erlang distributed with parameter 
$(k, \frac{1}{\theta} )$ with $k,\ \theta >0$, namely,
\begin{equation}\label{rmdG}
dF_{\Lambda} (\lambda) = \frac{1}{\Gamma(k) \theta^k}\lambda^{k-1}e^{-\frac{\lambda}{\theta} } \, d \lambda,
\end{equation}
we can derive explicitly the probability of default.
\begin{theorem}\label{thm-randomized-lambda}
The probability of default 
under \eqref{20180316-1} and \eqref{rmdG}
is expressed by 
\begin{equation}\label{randomlambda-2}
\begin{split}
\phi(u)
&=  \frac{\Gamma(k+\alpha +1)}{\Gamma(\alpha +1)\Gamma(k) }
B_{{\frac{u \theta }{c+u \theta}}} (\alpha +1, k), 
\end{split}
\end{equation}
where $B_x(u,v) = \int_0^x y^{u-1}(1-y)^{v-1} \, dy$ is the incomplete beta function.
\end{theorem}

\begin{proof}
One has that 
\begin{equation*}\label{randamlambda-1}
\begin{split}
\phi(u)
&= \int_0^{\infty} \phi_{\lambda} (u)\frac{1}{\Gamma(k) \theta^k}\lambda^{k-1}e^{-\frac{\lambda}{\theta} } \, d \lambda. 
\end{split}
\end{equation*}
Differentiating both sides, we have that 
\begin{equation*}
\begin{split}
\phi'(u)
&=  \frac{\Gamma(k+\alpha +1)}{\Gamma(\alpha +1)\Gamma(k) }
 \left(\frac{\theta}{c} \right)
\left(\frac{c  }{c+u \theta}\right)^{k+1}
 \left(1- \frac{c }{c+u \theta}\right)^{\alpha }
. 
\end{split}
\end{equation*}
Therefore for some constant $C$, the probability of default is given by 
\begin{equation*}
\begin{split}
\phi(u)
&
=  \frac{\Gamma(k+\alpha +1)}{\Gamma(\alpha +1)\Gamma(k) }\left(\frac{\theta}{c} \right) \int_0^u \left(\frac{c  }{c+x \theta}\right)^{k+1}
 \left(1- \frac{c }{c+x \theta}\right)^{\alpha} \, dx +C. 
\end{split}
\end{equation*}
Since $\phi_{\lambda} (0) = 0 $ for each $\lambda>0$, and $\phi (0)=0$, 
we have that $C=0$.
\if1
, that is, 
\begin{eqnarray*}
\begin{split}
\phi(u)
&=  \frac{\Gamma(k+\alpha +1)}{\Gamma(\alpha +1 )\Gamma(k) }\left(\frac{\theta}{c} \right)
 \int_0^u \left(\frac{c  }{c+x \theta}\right)^{k+1}
 \left(1- \frac{c }{c+x \theta}\right)^{\alpha} \, dx 
. 
\end{split}
\end{eqnarray*}
\fi 
By a further change of variables, we
\if1 
have that 
\begin{equation*}
\begin{split}
 \int_0^u \left(\frac{c  }{c+x \theta}\right)^{k+1}
 \left(1- \frac{c }{c+x \theta}\right)^{\alpha} \, dx 
= \frac{c}{\theta} 
B_{{\frac{u \theta }{c+u \theta}}} (\alpha +1, k).  
\end{split}
\end{equation*} 
Hence 
\fi 
conclude that the probability of default is expressed 
as in \eqref{randomlambda-2}. 
\end{proof}

\begin{remark}
We note that the order of $\phi(u)$ is $\mathcal{O} (1/u^k)$. 
\end{remark}

\section{Numerical experiments} \label{numerical}
{In this section we present firstly two algorithms for the calculating the probability of  default (ruin), apply them to some concrete numerical examples/chosen parameters, then analyse the corresponding premium rates.}
\subsection{Simulation for finite-time ruin versus infinite-time one}\label{Nclexp}
We present two algorithms. Algorithm 2 is faster and less variant than the Algorithm 1. However, Algorithm 2 can only be used for the case that inter-arrival is exponential distributed (or conditional exponential distributed for randomized arrival times model).
\subsection*{Algorithm 1}
First, an algorithm which simulates the default state up to a given finite horizon time $t$ by simulating scenarios of cash flows: 
\begin{align*}
U_t = c \sum_{n=1}^{\infty}\left[ 
1_{[T_n, T_{n+1})(t)}(t-T_n)e^{-\sum_{i=1}^n X_i}+
\sum_{i=1}^n (T_i-T_{i-1})e^{-\sum_{i=1}^{i-1} X_i}
\right].
\end{align*}
\begin{itemize}
    \item Initiate $k =0 $,  $T(0) = 0$, $X(0)=0$ as time and effect at time $0$.
    \item While $T(k)<t$
    \begin{itemize}
        \item $k = k+1$,
        \item Simulate $(W,X)$ given a specific distribution,
        \item Inter-arrival time $W(k) = W$,
        \item Arrival time $T(k) = T(k-1)+W(k)$,
        \item Effect $X(k)=X$.
    \end{itemize}
    \item Calculate the cash flow at $t$:$ U_t = (t-T_{k-1})e^{-\sum_{i=1}^{k-1} X(i)}+
\sum_{i=1}^{k-1} W(i){e^{-\sum_{i=1}^{i-1} X(i)}} $.
\item Determine default state $I = 1_{\{U_t<u\}}$
\end{itemize}
Repeat the procedure $N$ times (with $N$ large).
The sample average of $I$ is an estimation for the default probability $\phi(u,t)$.

\subsection*{Algorithm 2}
In the memoryless arrival model, the number of arrivals up to time $(T_n)_{n\geq 0}$ is a homogeneous Poisson process with parameter $\lambda$. The number of arrivals up to $t$, 
$ 
N_t = \sum_{n=1}^{\infty}1_{\{T_n\leq t\}}$
is also Poisson distributed with parameter $\lambda$. 
Given $N_t=n$, arrival times $T_1, \ldots, T_n$ is the ordered statistics of $n$ i.i.d $U_1, \ldots U_n$ with uniform distribution on $[0,t]$. 
So we can get the algorithm as follows:
\begin{itemize}
    \item Simulate a random number $n$ from Poisson distribution with parameter $\lambda t$.
    \item Simulate $n$ random number $U_1, ...,U_n$ from $n$ i.i.d uniform distribution $U([0,t])$.
    \item Sort $U_1, ...,U_n$ to get arrival time $T_1, \ldots, T_n$.
    \item Simulate effect $X_1, ..., X_n$ given a specific distribution.
     \item Calculate cash flow at $t$: 
     $ 
U_t = (t-T_{k-1})e^{-\sum_{i=1}^{k-1} X(i)}+
\sum_{i=1}^{k-1} W(i)
{ e^{-\sum_{i=1}^{i-1} X(i)}}$.
\item Determine the default state $I = 1_{\{U_t<u\}}$.
\end{itemize}
Repeat the procedure $N$ times (with $N$ large).
The expectation of $I$ is an estimation for default probability $\phi(u,t)$.
%
%

\color{black}
\subsection{Risk Adjusted Premium Rates}
As in \cite{Ragnar}, imposing a solvency level on $\psi(u)$ we could derive $c$. Comparing the resulting $c$ with the mortgage repayment rate, we could find out the amount that would be considered as premium.
 We are conducting a sensitivity analysis to see how each variable  impacts the results. 
The aim is to find the values of the parameters which minimise the associated risks, but still keep the mortgage attractive, i.e. keep the monthly payments relatively low. 
In real world the parameters $\alpha$ and $\lambda$ would be estimated using historical data, but for our analysis purposes, we will start from some given values.  

\subsection*{Memoryless Arrivals}
Let's assume \textit{X} follows an exponential distribution with parameter $\alpha$,  and W follows an exponential distribution with a parameter $\lambda$. \textit{X} and \textit{W} are independent of each other. Then, according to Proposition \ref{ex1}, for any \textit{u $\geq$ 0},  the probability of default is an incomplete Gamma function, \eqref{prop4-phi}.
%
{In our analysis, we focus on the repayment rate (with insurance premium included), as a percentage of the loan, namely $c/u$. When we impose a fixed solvency target, $\epsilon$, namely $\phi (u) \leq \epsilon$, from $
\frac{1}{\Gamma (\alpha+1)}\int_{0}^{\frac{\lambda u}{c}}z^{\alpha}e^{-z}dz \leq \epsilon $ we obtain that $\frac{c}{u} \geq \frac{\lambda}{\Gamma_{\alpha}^{-1} (\epsilon)} $, 
where $\Gamma_{\alpha}^{-1}$ is the inverse function of  regularized incomplete Gamma with parameter $\alpha$. 
Table \ref{table:c-u} presents the minimum values of   $\frac{c}{u}=\frac{\lambda}{\Gamma_{\alpha}^{-1} (\epsilon) },$ for $\epsilon= 0.00001$, when varying the parameters $\alpha $ and $\lambda$. }
\begin{table}[htbp]
\caption{$\frac{c}{u}$ for different values of the parameters $\lambda$ and $\alpha$} 
\centering 	 
\begin{tabular}{c|cccccc} 				
\hline\hline 											
$\lambda \backslash \alpha$ &1&2&3&4&5&6	  \\ 	\hline
1&
0.0850603 & 0.0717969 & 0.0628385 & 0.0562366 & 0.0511059 & 0.0469712  \\
2&
0.170121 & 0.143594 & 0.125677 & 0.112473 & 0.102212 & 0.0939424  \\
3&0.255181 & 0.215391 & 0.188515 & 0.16871 & 0.153318 & 0.140914 \\
4&0.340241 & 0.287188 & 0.251354 & 0.224946 & 0.204424 & 0.187885 \\
5&0.425301 & 0.358985 & 0.314192 & 0.281183 & 0.25553 & 0.234856  \\
\hline\hline
\end{tabular}
\label{table:c-u} 
\end{table}
Moreover, one can simulate histograms of the cash flow at a specific times. For instance, for $\lambda = 0.5$ and $\alpha = 20$, the histograms at $T= 100$ (Figure \ref{fig2} and Figure \ref{fig3}), present very small differences,  which are caused by the different estimates of the  probability of default. 
\begin{figure}[H]
    \begin{tabular}{cc}
      \begin{minipage}[t]{0.45\hsize}
        \centering
    \includegraphics[width=1\linewidth]{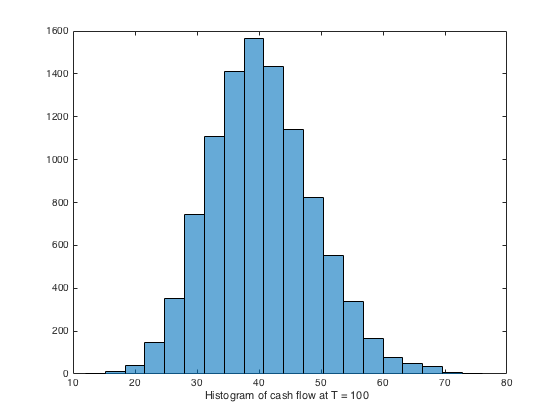}
    \caption{Histogram for cash flow at a given time by Algorithm 1} 
    \label{fig2}
      \end{minipage} &
      \begin{minipage}[t]{0.45\hsize}
        \centering
        \includegraphics[width=1\linewidth]{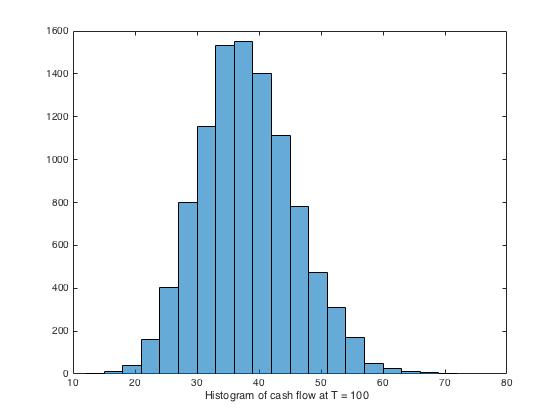}
    \caption{Histogram for cash flow at a given time by Algorithm 2} 
    \label{fig3}
  \end{minipage}
    \end{tabular}
  \end{figure}

\if1 
\begin{figure}[htbp]
    \centering
    \includegraphics[width=0.5\linewidth]{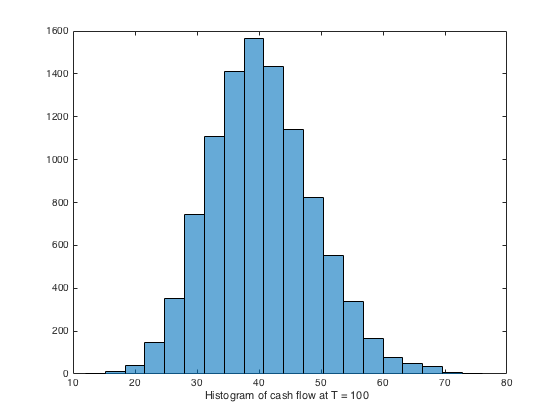}
    \caption{Histogram for cash flow at a given time by Algorithm 1} 
    \label{fig2}
\end{figure}
\begin{figure}[htbp]
    \centering
    \includegraphics[width=0.5\linewidth]{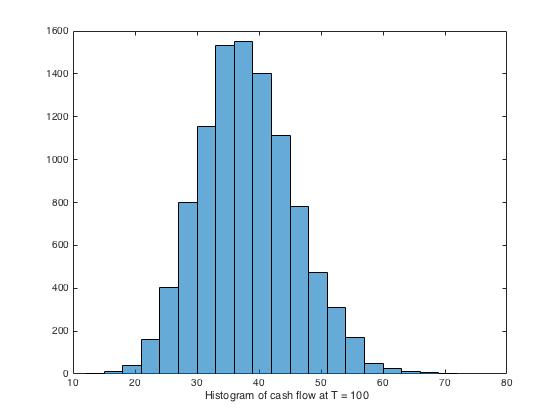}
    \caption{Histogram for cash flow at a given time by Algorithm 2} 
    \label{fig3}
\end{figure}

\begin{figure}[htbp]
\centering
\includegraphics[width =0.5\linewidth]{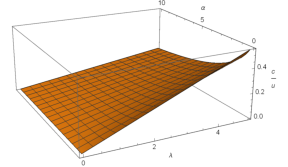}
\caption{Graph of table 1 }
\label{figure:c-u} 
\end{figure}
\fi 
\subsection*{Randomised Arrivals}
Recall from Theorem \ref{thm-randomized-lambda}, that for any
$u \geq 0$,
the probability of default is {an incomplete Beta function, \eqref{randomlambda-2}.}
We {consider $\Lambda$, such that} $E[\Lambda]=k \theta$ and $V[\Lambda] = k \theta^2$. 
When {we impose a solvency target $\phi (u) \leq \epsilon$} and when $k \theta =1\ (\theta = 1/k)$,  then  from Theorem \ref{thm-randomized-lambda},
$
\frac{c}{u} \geq \frac{(1- B^{-1}_{\alpha +1, k} (\epsilon)) }{k B^{-1}_{\alpha +1, k} (\epsilon)} $, 
 where $B^{-1}_{\alpha +1, k}$ is the inverse function of  regularised incomplete Beta with parameters $\alpha+1$ and $k$. %
Table \ref{table:c-u-l} describes the minimum values realized by $c/u$, for given parameters $\alpha $ and $k$, under a solvency target $\epsilon=0.0001.$
\begin{table}[htbp]
\caption{$\frac{c}{u}$ for different values of the parameters $k$ and $\alpha$} 
\centering 											
\begin{tabular}{c|ccccccc} 								
\hline\hline 											
$k \backslash \alpha$ &1&2&3&4&5&6&7  \\ 	\hline
1&99. & 20.5443 & 9. & 5.30957 & 3.64159 & 2.72759 & 2.16228  \\
3&81.0935 & 14.9698 & 6.08913 & 3.41611 & 2.26003 & 1.64776 & 1.27925 \\
5&76.9924 & 13.6923 & 5.42064 & 2.98014 & 1.94112 & 1.39791 & 1.07451 \\
7&75.1637 & 13.1197 & 5.11958 & 2.78298 & 1.79635 & 1.28412 & 0.981004\\
9&74.1276 & 12.794 & 4.94774 & 2.67007 & 1.7132 & 1.21859 & 0.927026 \\
\hline\hline
\end{tabular}
\label{table:c-u-l} 
\end{table}
\if1 
\begin{figure}[htbp]
\centering
\includegraphics[width = 0.5\linewidth]{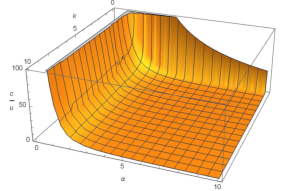}
\caption{Graph of table 2}
\label{figure:c-u-l} 
\end{figure}
\fi 
For inter-arrival times $W$ exponentially distributed with parameter $\lambda = 0.2$, $X$ is exponentially distributed with parameter $\alpha =20$,
and $\frac{u}{c} = 50$, Figure \ref{fig1} provides information about the convergence behavior of the default probability, as the horizon time tends to infinity. 
After $N=10^5$ simulations, one can see that the probability of default in finite time horizon converges exponentially to the one in infinite time horizon, namely, in this particular example, to $0.0035.$ (Figure \ref{fig3+}).
\begin{figure}[H]
    \begin{tabular}{cc}
      \begin{minipage}[t]{0.45\hsize}
        \centering
    \includegraphics[width=1\linewidth]{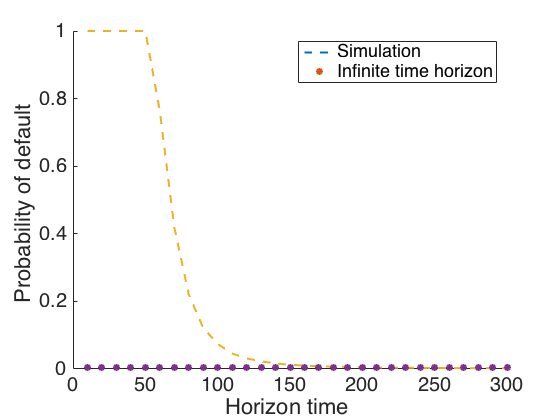}
    \caption{A comparison of probability of default by simulation} 
    \label{fig1}
      \end{minipage} &
      \begin{minipage}[t]{0.45\hsize}
        \centering
    \includegraphics[width=1\linewidth]{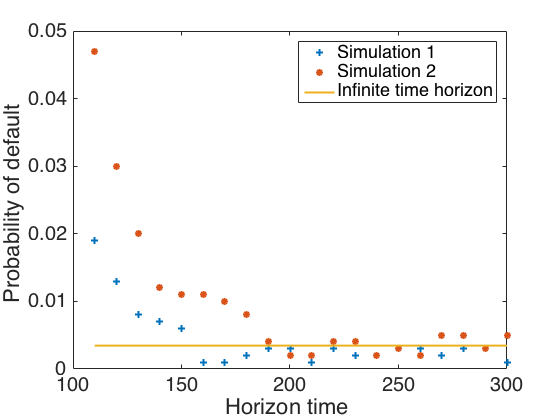}
    \caption{Convergence of probability of default from finite horizon time to infinite horizon time} 
   \label{fig3+}
  \end{minipage}
    \end{tabular}
  \end{figure}

\if1 
\begin{figure}
    \centering
    \includegraphics[width=0.5\linewidth]{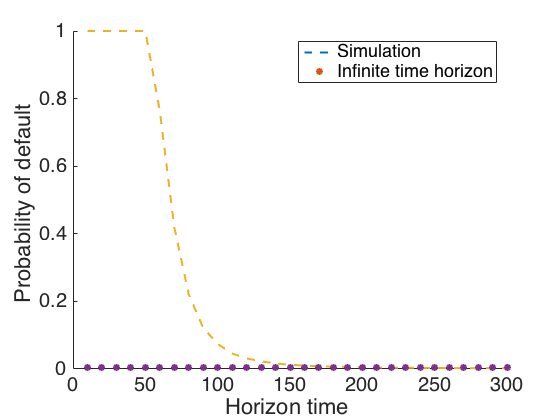}
    \caption{A comparison of probability of default by simulation} 
    \label{fig1}
\end{figure}
 
\begin{figure}
   \centering
    \includegraphics[width=0.5\linewidth]{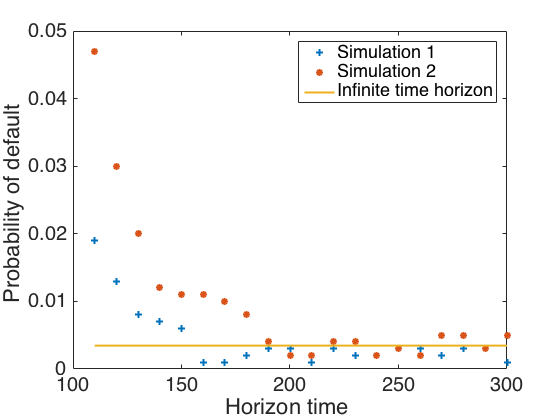}
    \caption{Convergence of probability of default from finite horizon time to infinite horizon time} 
   \label{fig3+}
\end{figure}
\fi 
\subsection*{Comparison 1}
For $\lambda >0$, we assume that $k = 1$ and $\theta = \lambda$, meaning $E[\Lambda]=\lambda$ and $Var[\Lambda]=\lambda$. 
The graph \ref{figure:lambda-phi} 
show 
the difference between the  probability of loan repayment in the Memoryless Arrivals (MA) case 
versus the Randomize Arrivals (RA) case, when $\alpha  = 20$, 
$c = 2$ and $u = 100$. 
Here the probabilities of default are given by \eqref{prop4-phi} and \eqref{randomlambda-2}, respectively.


\subsection*{Comparison 2}
For $\lambda >0$ and $v >0$, we assume that $k = \frac{\lambda^2}{v} $ and $\theta = \frac{v}{\lambda}$, meaning $E[\Lambda]=\lambda$ and $Var[\Lambda]=v$. In other words,  the variance of $\Lambda$ is fixed. 
The graph \ref{figure:lambda-phi2} 
shows 
the difference of the default probabilities  in the  MA  case 
versus the RA case, when $\alpha  = 20$, 
$c = 2$ and $u = 100$.

\begin{figure}[H]
    \begin{tabular}{cc}
      \begin{minipage}[t]{0.45\hsize}
        \centering
\includegraphics[width = 1\linewidth]{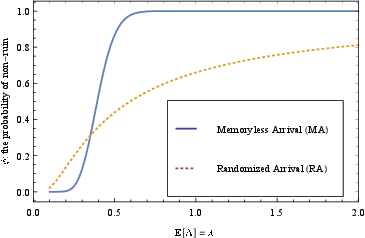}
\caption{Default probabilities}
\label{figure:lambda-phi} 
      \end{minipage} &
      \begin{minipage}[t]{0.45\hsize}
        \centering
\includegraphics[width=1\linewidth]{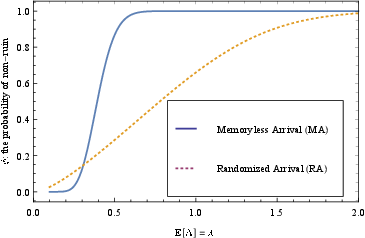}
\caption{Default probabilities}
\label{figure:lambda-phi2} 
 \end{minipage}
    \end{tabular}
  \end{figure}

\if1 
\begin{figure}[htbp]
\centering
\includegraphics[width = 1\linewidth]{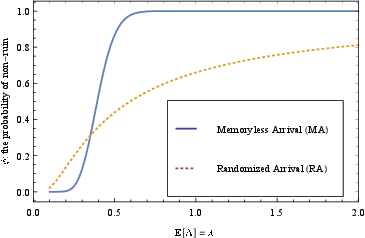}
\caption{Default probabilities}
\label{figure:lambda-phi} 
\end{figure}

\begin{figure}[htbp]
\centering
\includegraphics[width=0.5\linewidth]{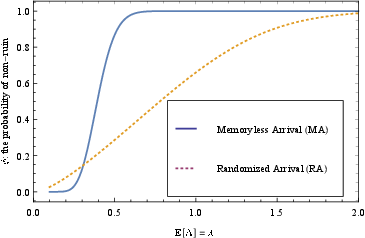}
\caption{Default probabilities}
\label{figure:lambda-phi2} 
\end{figure}
\fi 

%

\section{Conclusion} \label{conclusion}
Although earthquakes are not rare events in Japan, for cultural reasons,  only a small part of home-owners in Japan have the relevant insurance cover.  Thus in the event of a disaster produced by an earthquake,  the uninsured home-owners, roughly 70\% of the market, could end up with a completely demolished property and an outstanding mortgage loan (e.g in 2009 only 23\% had earthquake insurance), the double-debt problem. The paper proposes a solution to this problem specific to Japanese mortgages, via an insurance mechanism incorporated in the mortgage repayment scheme. The premium of the insurance can be determined from the probability of default, by solvency control (i.e. less than a given/impose upper bound). By deriving the probability of default of the loan provider in infinte-time, one can then infer close approximations for any finite-time horizons. The modeling framework is that of exponential functional of renewal-reward processes and the methodology stems from mathematical risk theory. The set-up and the results could be translated in other financial or actuarial applications. For instance, as in \cite{Pflug}, the exponential functional of a reward process can model the wealth of one person/ family (in danger of down-crossing the poverty line) with an insurance proposal to protect the vulnerable (close to the poverty line) from falling into {\it poverty traps} (which are absorbing states).  \\
\section*{Acknowledgments}
The authors would like to thank Professor Ohgaki for fruitful discussions during the preparation of this paper. Also, many thanks to BNP Paribas-Cardiff for their input regarding the practicality of the model.
%
\bibliographystyle{plain}
\bibliography{bibliography}
%
\appendix 
\section{Proof of \eqref{thm-l1-Kn-rec}}
The aim of this section is to prove \eqref{thm-l1-Kn-rec}, 
that is, 
\begin{equation*} 
\begin{split}
&\int_0^\infty  |
\mathcal{K}_{\infty}^n\psi_{0}^{\infty}(u) 
-\mathcal{K}^n\psi_0(u,t)| 
\, du 
\\& \leq
\mathbf{E}[e^{-X}]
\mathbf{E}[I_{\{W<t\}}\int_0^{\infty}
 | \mathcal{K}_{\infty}^{n-1}\psi_{0}^{\infty}(u)  - \mathcal{K}^{n-1}\psi_{0}  (u, t-W) |\, du]
\\& \qquad  + c 
\mathbf{E}[e^{-X}]^n \mathbf{P}(W > t) 
\mathbf{E}[ W ]
. \end{split}
\end{equation*}
To obtain \eqref{thm-l1-Kn-rec}, we need the following Lemmas. 
\begin{lemma}\label{Lem-Knh-26}
For $h \in \mathcal{A}_1$, $ t > 0$ and $ n \in \mathbf{N} \cup \{0\}$, it holds that  
\begin{equation}\label{eq-lem-Kn-26}
\begin{split}
\int_0^{\infty} 
\mathcal{K} h (u,t )
\, du  = \mathbf{E}[e^{-X} ] \mathbf{E}[ I_{\{W<t\}} \int_0^{\infty} h(z,t - W ) \, dz ]. \end{split}
\end{equation}
\end{lemma}
\begin{proof}
By the definition of $\mathcal{K}$,
\begin{equation}\label{lem-gKL1-26}
 \begin{split}
\int_0^\infty\mathcal{K}h  (u,t)  \, du 
= 
\int_0^\infty \, du 
\int_0^{t \wedge \frac{u}{c}} \, dw 
\int_0^\infty \,dx \, 
h(e^x( u-cw) ,t-w ) 
f_{W,X}(w,x)
. 
\end{split}
\end{equation}
Here we note that $W$ and  $X$ are non-negative valued random variables. 
By change of the variables $
(u, w,x ) \mapsto (u,w, z =e^x( u-cw) ) $, 
\begin{equation*}
\begin{split}
&\mbox{(the right hand side of \eqref{lem-gKL1-26}) }
\\& = 
\int_0^{t} \, dw 
\int_{cw}^\infty \, du 
\int_{u-cw}^\infty \,dz \, 
h(z ,t-w ) 
f_{W,X}(w,\log \frac{z}{u-cw} ) \frac{1}{z}
\\& = 
\int_0^{t} \, dw 
\int_{0}^\infty \,dz \, 
\int_{cw}^{z+ cw} \, du \, 
h(z ,t-w ) 
f_{W,X}(w,\log \frac{z}{u-cw} ) \frac{1}{z}
. \end{split}
\end{equation*}
By changing variables   $(w,z,u) \mapsto (w,z, v= \log \frac{z}{u-cw}) $, we conclude equation \eqref{eq-lem-Kn-26}. 
\end{proof} 
\begin{lemma}\label{Lem-Kng-26}
For $g \in L_1$ and $ n \in \mathbf{N} \cup \{0\}$, it holds that  
\begin{equation}\label{eq-lem-Kgn-26}
\begin{split}
\int_0^{\infty} 
\mathcal{K}^n_{\infty} g (u )
\, du  = \mathbf{E}[e^{-X} ]^n  \int_0^{\infty} g(z ) \, dz 
. \end{split}
\end{equation}
In particular, it holds that 
\begin{equation*}
\int_0^{\infty} 
\mathcal{K}^n_{\infty} \psi_0^\infty(u)
\, du  
=c \mathbf{E}[e^{-X} ]^n\mathbf{E}[W]
. 
\end{equation*}
\end{lemma}
\begin{proof}
For $n=0$, clearly \eqref{eq-lem-Kgn-26} holds. 
For $n \geq 1$, we see that 
\begin{equation}\label{eq-lem-Kng1-26}
\begin{split}
\int_0^\infty\mathcal{K}^n_{\infty} g  (u)  \, du 
 = 
\mathbf{E}[e^{-X}]
\int_0^{\infty}
\mathcal{K}^{n-1}_{\infty} g (u)\, du
.\end{split}
\end{equation}
Hence we conclude \eqref{eq-lem-Kgn-26} by induction. 
\end{proof} 
Now we give a proof of \eqref{thm-l1-Kn-rec}.
Let us consider a decomposition of  $\mathcal{K}_{\infty}\psi_{0}^{\infty}(u) $. 
By the definition of $\mathcal{K}$ and $\mathcal{K}_{\infty}$, for $n \in \mathbf{N}$, we see that 
\begin{equation*}\label{thm-n-eq0-1-26}
\begin{split}
&\mathcal{K}_{\infty}^n\psi_{0}^{\infty}(u) 
\\& = 
I_{\{t \leq \frac{u}{c} \}}\mathbf{E} [ ( I_{\{W \leq t\}} +  I_{\{t \leq W \leq \frac{u}{c} \}} )
\mathcal{K}_{\infty}^{n-1}\psi_{0}^{\infty}(e^X (u-cW)) ] 
\\& \qquad + 
I_{\{\frac{u}{c} < t \}}
\mathbf{E} [  I_{\{ W \leq \frac{u}{c} \}} 
\mathcal{K}_{\infty}^{n-1}\psi_{0}^{\infty}(e^X (u-cW)) ] 
\\& = \left(I_{\{\frac{u}{c} < t \}}\mathbf{E} [  I_{\{ W \leq \frac{u}{c} \}} 
\mathcal{K}_{\infty}^{n-1}\psi_{0}^{\infty}(e^X (u-cW)) ] 
+I_{\{\frac{u}{c} < t \}}
\mathbf{E} [  I_{\{ W \leq \frac{u}{c} \}} 
 \mathcal{K}_{\infty}^{n-1}\psi_{0}^{\infty}(e^X (u-cW)) ]\right)
 \\&\qquad + \mathbf{E} [ I_{\{t \leq W \leq \frac{u}{c} \}} )
\mathcal{K}_{\infty}^{n-1}\psi_{0}^{\infty}(e^X (u-cW)) ] 
\\& = 
\mathcal{K} \mathcal{K}_{\infty}^{n-1}\psi_{0}^{\infty}(u,t) 
+ I_{\{t \leq \frac{u}{c} \}}\mathbf{E} [ I_{\{t \leq W \leq \frac{u}{c} \}}
\mathcal{K}_{\infty}^{n-1}\psi_{0}^{\infty}(e^X (u-cW)) ] 
. \end{split}
\end{equation*}
Here $\mathcal{K}_{\infty}^{n-1}\psi_{0}^{\infty}(e^X (u-cW)) $ is identified with $\mathcal{K}_{\infty}^{n-1}\psi_{0}^{\infty}(e^X (u-cW), t-W) $. 
Since the operator $\mathcal{K}$ is linear, 
we have that 
\begin{equation}\label{thm-n-eq0-26}
\begin{split}
&\mathcal{K}_{\infty}^n\psi_{0}^{\infty}(u) -\mathcal{K}^n\psi_{0}(u, t)
\\&= \mathcal{K}  ( \mathcal{K}_{\infty}^{n-1}\psi_{0}^{\infty}(u,t)  - \mathcal{K}^{n-1}\psi_{0}(u,t)  ) 
+ I_{\{t \leq \frac{u}{c} \}}\mathbf{E} [ I_{\{t \leq W \leq \frac{u}{c} \}}
\mathcal{K}_{\infty}^{n-1}\psi_{0}^{\infty}(e^X (u-cW)) ] 
. \end{split}
\end{equation}
Therefore we obtain that 
\begin{equation}\label{thm-n-eq3-26} 
\begin{split}
&\int_0^{\infty} | \mathcal{K}_{\infty}^n\psi_{0}^{\infty}(u) 
-\mathcal{K}^n\psi_0(u,t) | \, du 
\\& \leq 
\int_0^{\infty} 
\mathcal{K}  | \mathcal{K}_{\infty}^{n-1}\psi_{0}^{\infty} - \mathcal{K}^{n-1}\psi_{0} | (u,t)  
\, du 
\\& \qquad 
+\int_0^{\infty} 
 I_{\{t \leq \frac{u}{c} \}}\mathbf{E} [ I_{\{t \leq W \leq \frac{u}{c} \}}
\mathcal{K}_{\infty}^{n-1}\psi_{0}^{\infty}(e^X (u-cW)) ]  \, du 
. \end{split}
\end{equation}
By Fubini's theorem, the second term of the right hand side in \eqref{thm-n-eq0-26} is 
\begin{equation}\label{thm-n-eq1-26}
\begin{split}
\mathbf{E}[e^{-X}] \mathbf{P}(W > t) 
\int_0^{\infty}
 \mathcal{K}_{\infty}^{n-1}\psi_{0}^{\infty} (u)\, du
. \end{split}
\end{equation}
Hence by Lemma \ref{Lem-Knh-26}, Lemma \ref{Lem-Kng-26} and the combination of \eqref{thm-n-eq3-26} and \eqref{thm-n-eq1-26}, we conclude \eqref{thm-l1-Kn-rec}.

\end{document}